\documentclass[11pt]{amsart}
\usepackage{amsmath,amsthm,amsfonts,amssymb,mathrsfs, mathtools}
\usepackage{enumerate}
\usepackage[colorlinks=true, linkcolor=blue, citecolor=magenta, menucolor=black]{hyperref}
\usepackage[demo]{graphicx}
\usepackage[up]{caption}
\usepackage{subcaption}
\usepackage{pgf,tikz}
\usepackage[toc,page]{appendix}
\usetikzlibrary{arrows}
\usetikzlibrary{patterns}
\usepackage{color}

\numberwithin{equation}{section}
\theoremstyle{plain}
\newtheorem{thm}{Theorem}

\newtheorem{lem}[thm]{Lemma}

\newtheorem{letterthm}{Theorem}

\newtheorem{lettercor}[letterthm]{Corollary}

\theoremstyle{definition}

\newtheorem*{defn*}{Definition}

\newtheorem*{CRC}{Connes'\! rigidity conjecture}

\newcommand{\R}{\mathbb{R}}
\newcommand{\C}{\mathbb{C}}

\newcommand{\mult}{\operatorname{mult}}

\newcommand{\ovt}{\mathbin{\overline{\otimes}}}

\newcommand{\id}{\operatorname{id}}

\newcommand{\Prob}{\operatorname{Prob}}
\newcommand{\supp}{\operatorname{supp}}

\newcommand{\rk}{\operatorname{rk}}

\newcommand{\rB}{\operatorname{B}}

\newcommand{\rE}{\operatorname{ E}}
\newcommand{\rC}{\operatorname{C}}

\newcommand{\rL}{\operatorname{ L}}

\newcommand{\Har}{\operatorname{Har}}

\allowdisplaybreaks

\makeatletter
\@namedef{subjclassname@2020}{%
  \textup{2020} Mathematics Subject Classification}
\makeatother

\begin{document}

\title[OA characterization of the nc Poisson boundary]{Operator algebraic characterization of the noncommutative Poisson boundary}

\begin{abstract}
We obtain an operator algebraic characterization of the noncommutative Furstenberg--Poisson boundary $\rL(\Gamma) \subset \rL(\Gamma \curvearrowright B)$ associated with an admissible probability measure $\mu \in \Prob(\Gamma)$ for which the $(\Gamma, \mu)$-Furstenberg--Poisson boundary $(B, \nu_B)$ is uniquely $\mu$-stationary. This is a noncommutative generalization of Nevo--Sageev's structure theorem \cite{NS11}. We apply this result in combination with previous works to provide further evidence towards Connes'\! rigidity conjecture for higher rank lattices.
\end{abstract}

\author{Cyril Houdayer}
\address{\'Ecole Normale Sup\'erieure \\ D\'epartement de Math\'ematiques et Applications \\ Universit\'e Paris-Saclay \\ 45 rue d'Ulm \\ 75230 Paris Cedex 05 \\ France}
\email{cyril.houdayer@ens.psl.eu}
\thanks{CH is supported by ERC Advanced Grant NET 101141693}

\subjclass[2020]{20G25, 22D25, 46L10, 46L55}
\keywords{Connes'\! rigidity conjecture; Higher rank lattices; Noncommutative Furstenberg--Poisson boundaries; von Neumann algebras}

\maketitle

\section{Introduction and statement of the main results}

Let $\Gamma$ be a countable discrete group and $\mu \in \Prob(\Gamma)$ an admissible probability measure in the sense that $\bigcup_{n \geq 1} \supp(\mu)^n = \Gamma$. Denote by $(B, \nu_B)$ the $(\Gamma, \mu)$-Furstenberg--Poisson boundary \cite{Fu62, BS04}. Recall that $(B, \nu_B)$ is the unique $(\Gamma, \mu)$-space, up to isomorphism, for which the $\Gamma$-equivariant Poisson transform 
$$\mathscr P_\mu : \rL^\infty(B, \nu_B) \to \Har^\infty(\Gamma, \mu) : f \mapsto \left( \gamma \mapsto \int_B f(\gamma b) \, {\rm d}\nu_B(b) \right)$$ is onto and isometric. Here, $\Har^\infty(\Gamma, \mu) \subset \ell^\infty(\Gamma)$ denotes the subspace of bounded (right) $\mu$-harmonic functions. We say that the $(\Gamma, \mu)$-Furstenberg--Poisson boundary $(B, \nu_B)$ is $\mu$-{\em uniquely stationary} if there exists a compact metrizable model $K$ of $(B, \nu_B)$ for which there exists a unique $\mu$-stationary Borel probability measure $\nu_{K} \in \Prob(K)$ and $(K, \nu_{K}) \cong (B, \nu_B)$ as $(\Gamma, \mu)$-spaces. In \cite[Theorem 9.2]{NS11}, Nevo--Sageev showed that if the $(\Gamma, \mu)$-Furstenberg--Poisson boundary $(B, \nu_B)$ is $\mu$-uniquely stationary, then for any amenable $(\Gamma, \mu)$-space $(X, \nu_X)$, there exists a relatively measure preserving $\Gamma$-equivariant measurable factor map $(X, \nu_X) \to (B, \nu_B)$.

Denote by $M = \rL(\Gamma)$ the group von Neumann algebra and by $\mathscr B = \rL(\Gamma \curvearrowright B)$ the group measure space construction associated with the nonsingular action $\Gamma \curvearrowright (B, \nu_B)$. Then we have the natural inclusion $M \subset \mathscr B$. We also consider the conjugation action $\Gamma \curvearrowright \mathscr B$ and we denote by $\rE_{B} : \mathscr B \to \rL^\infty(B, \nu_B)$ the canonical $\Gamma$-equivariant faithful normal conditional expectation. Then the faithful normal state $\psi_B = \nu_B \circ \rE_B \in \mathscr B_\ast$ is $\mu$-stationary. Regard $\rL^2(M, \tau) \subset \rL^2(\mathscr B, \psi_B)$ as a closed subspace and simply denote by $e_M : \rL^2(\mathscr B, \psi_B) \to \rL^2(M, \tau)$ the orthogonal projection. Define the normal state $\varphi_\mu \in \rB(\ell^2(\Gamma))_\ast$ by the formula $\varphi_\mu(T) = \sum_{\gamma \in \Gamma} \mu(\gamma) \langle T \delta_\gamma, \delta_\gamma \rangle$ for $T \in \rB(\ell^2(\Gamma))$. Define the normal ucp (unital completely positive) map
$$\Phi_\mu : \rB(\ell^2(\Gamma)) \to \rB(\ell^2(\Gamma)) : T \mapsto \sum_{\gamma \in \Gamma} \mu(\gamma) J\lambda(\gamma)J  T  J \lambda(\gamma)^* J.$$
Then we have $\varphi_\mu(T) = \langle \Phi_\mu(T)\delta_e, \delta_e \rangle$ for all $T \in \rB(\ell^2(\Gamma))$. Denote by $\Har(\Phi_\mu) = \{T \in \rB(\ell^2(\Gamma)) \mid \Phi_\mu(T) = T\}$ the subspace of $\Phi_\mu$-harmonic elements. Following \cite{Iz04, DP20}, the noncommutative Poisson transform
$$\widehat{\mathscr P}_\mu : \mathscr B \to \Har(\Phi_\mu) : T \mapsto e_M T e_M$$
is a normal onto isometric ucp map such that $\widehat{\mathscr P}_\mu|_M = \id_M$. Then we refer to the inclusion $M \subset \mathscr B$ as the $(M, \varphi_\mu)$-noncommutative Furstenberg--Poisson boundary.

Let $M \subset \mathscr M$ be an inclusion of von Neumann algebras and consider the conjugation action $\Gamma \curvearrowright \mathscr M$. By \cite[Proposition 4.2]{BBHP20}, to any (faithful) normal $\mu$-stationary state $\psi \in \mathscr M_\ast$ corresponds a unique $\Gamma$-equivariant (faithful) normal ucp map $\Theta : \mathscr M \to \rL^\infty(B, \nu_B)$ such that $\nu_B \circ \Theta = \psi$. Moreover, by \cite[Proposition 2.1]{DP20}, to any (faithful) normal $\mu$-stationary state $\psi \in \mathscr M_\ast$ corresponds a unique (faithful) normal ucp map $\widehat \Theta : \mathscr M \to \mathscr B$ such that $\widehat \Theta|_{M} = \id_M$, $\psi = \nu_B \circ \rE_B \circ \widehat\Theta$ and $\Theta = \rE_B \circ \widehat \Theta$.

Whenever $\Phi : A \to B$ is a ucp map between unital $\rC^*$-algebras, we denote by $\mult(\Phi) \subset A$ the multiplicative domain of $\Phi$. Our main result is the following noncommutative generalization of Nevo--Sageev's structure theorem \cite[Theorem 9.2]{NS11}.

\begin{letterthm}\label{thm-NCNS}
Keep the same notation as above with $M = \rL(\Gamma)$ and $\mathscr B = \rL(\Gamma \curvearrowright B)$. Assume that the $(\Gamma, \mu)$-Furstenberg--Poisson boundary $(B, \nu_B)$ is $\mu$-uniquely stationary. Let $M \subset \mathscr M$ be an inclusion of von Neumann algebras. Assume that $\mathscr M$ is amenable and that there exists a faithful normal ucp map $\widehat\Theta: \mathscr M \to \mathscr B$ such that $\widehat\Theta|_M = \id_M$. 

Then $\widehat\Theta|_{\mult(\widehat\Theta)} : \mult(\widehat\Theta) \to \mathscr B$ is a unital normal onto $\ast$-isomorphism. Therefore, we may regard $\mathscr B \subset \mathscr M$ as a von Neumann subalgebra and $\widehat \Theta: \mathscr M \to \mathscr B$ as a faithful normal conditional expectation.
\end{letterthm}

In particular, letting $\psi = \nu_B \circ \rE_B \circ \widehat \Theta \in \mathscr M_\ast$ and $\psi_B = \nu_B \circ \rE_B \in \mathscr B_\ast$, it follows from \cite[Theorem A]{Zh23} and \cite{KV82} that the inclusion $M \subset \mathscr M$ has maximal $\varphi_\mu$-entropy in the sense of \cite{DP20}, that is, $$h_{\varphi_\mu}(M \subset \mathscr M, \psi) = h_{\varphi_\mu}(M \subset \mathscr B, \psi_B) = h(\mu) = \lim_n \frac{1}{n} H(\mu^{\ast n}).$$

When $\Gamma$ is moreover an infinite icc group (infinite conjugacy classes), we deduce the following operator algebraic characterization of the $(M, \varphi_\mu)$-noncommutative Furstenberg--Poisson boundary $M \subset \mathscr B$.

\begin{lettercor}\label{cor-NCNS}
Keep the same notation as above with $M = \rL(\Gamma)$ and $\mathscr B = \rL(\Gamma \curvearrowright B)$. Assume that $\Gamma$ is infinite icc and that the $(\Gamma, \mu)$-Furstenberg--Poisson boundary $(B, \nu_B)$ is $\mu$-uniquely stationary. Let $M \subset \mathscr M$ be an inclusion of von Neumann algebras satisfying the following conditions:
\begin{itemize}
\item [$(\rm i)$] $M' \cap \mathscr M = \C1$.
\item [$(\rm ii)$] There exists a normal ucp map $\widehat \Theta : \mathscr M \to \mathscr B$ such that $\widehat\Theta|_{M} = \id_{M}$.
\item [$(\rm iii)$] $\mathscr M$ is amenable.
\item [$(\rm iv)$] Whenever $M \subset \mathscr N \subset \mathscr M$ is an intermediate von Neumann subalgebra for which there exists a normal conditional expectation $\rE : \mathscr M \to \mathscr N$, we have $\mathscr N = \mathscr M$.
\end{itemize}

Then $\widehat \Theta : \mathscr M \to \mathscr B$ is a unital normal onto $\ast$-isomorphism.
\end{lettercor}

Let us point out that the $(M, \varphi_\mu)$-noncommutative Furstenberg--Poisson boundary $M \subset \mathscr B$ satisfies Conditions $(\rm i)$, $(\rm ii)$, $(\rm iii)$, $(\rm iv)$ by \cite{DP20}. Thus, Corollary \ref{cor-NCNS} gives an abstract operator algebraic characterization of the inclusion $M \subset \mathscr B$. 

We also apply Theorem \ref{thm-NCNS} to provide further evidence towards Connes'\! rigidity conjecture for higher rank lattices. 

\begin{CRC}
For every $i \in \{1, 2\}$, let $G_i$ be a semisimple connected real Lie group with trivial center, no compact factor such that $\rk_{\R}(G_i) \geq 2$ and let $\Gamma_i < G_i$ be an irreducible lattice. If $\rL(\Gamma_1) \cong \rL(\Gamma_2)$, then $G_1 \cong G_2$ and in particular $\rk_{\R}(G_1) = \rk_{\R}(G_2)$.
\end{CRC}

For every $i \in \{1, 2\}$, choose a Furstenberg measure $\mu_i \in \Prob(\Gamma_i)$ so that $(G_i/P_i, \nu_{P_i})$ is the $(\Gamma_i, \mu_i)$-Furstenberg--Poisson boundary \cite{Fu67}. Here, $\nu_{P_i} \in \Prob(G_i/P_i)$ denotes the unique $K_i$-invariant Borel probability measure, where $K_i < G_i$ is a maximal compact subgroup and $G_i = K_i P_i$. It is well-known that $(G_i/P_i, \nu_{P_i})$ is $\mu_i$-uniquely stationary \cite{GM89}. Set $M_i = \rL(\Gamma_i)$ and $\mathscr B_i = \rL(\Gamma_i \curvearrowright G_i/P_i)$.

Assume that $M_1 \cong M_2$ and set $M = M_1 = M_2$. Since $\mathscr B_2$ is an amenable (hence injective) von Neumann algebra and since $M \subset \mathscr B_1$, Arveson's extension theorem implies that there exists a ucp map $\Phi : \mathscr B_1 \to \mathscr B_2$ such that $\Phi|_M = \id_M$. We show that if $\Phi : \mathscr B_1 \to \mathscr B_2$ is normal, then $\rk_\R(G_1) = \rk_\R(G_2)$. 

\begin{letterthm}\label{thm-CRC}
Keep the same notation as above. Assume that there exists a normal ucp map $\Phi : \mathscr B_1 \to \mathscr B_2$ such that $\Phi|_M = \id_M$. Then $\Phi : \mathscr B_1 \to \mathscr B_2$ is a unital normal onto $\ast$-isomorphism. In particular, we have 
$$\rk_\R(G_1) = \rk_\R(G_2).$$
\end{letterthm}

\subsection*{Acknowledgments} I am grateful to Amos Nevo for thought provoking discussions during the IHP trimester {\em Group actions and rigidity:\! around the Zimmer program} and for pointing out the reference \cite{NS11}.

\section{Proofs of the main results}

We record the following well-known fact on ucp maps.

\begin{lem}\label{lem-useful1}
Let $\mathfrak A, \mathfrak B, \mathfrak C$ be unital $\rC^*$-algebras and $\Phi : \mathfrak A \to \mathfrak B$, $\Psi : \mathfrak B \to \mathfrak C$ be ucp maps. Assume that $\Psi$ is faithful and that $\Psi \circ \Phi$ is a unital $\ast$-homomorphism. Then $\Phi$ is a unital $\ast$-homomorphism and $\Phi(\mathfrak A) \subset \mult(\Psi)$.
\end{lem}

\begin{proof}
By Kadison's inequality and since $\Psi \circ \Phi$ is a unital $\ast$-homomorphism, for every $a \in \mathfrak A$, we have
$$\Psi(\Phi(a^*a)) = \Psi(\Phi(a^*))\Psi(\Phi(a)) \leq \Psi(\Phi(a)^* \Phi(a)) \leq \Psi(\Phi(a^*a)).$$
Since $\Psi$ is faithful, it follows that $\Phi(a)^* \Phi(a) = \Phi(a^*a)$. Then $\mult(\Phi) = \mathfrak A$ and so $\Phi$ is a unital $\ast$-homomorphism. Moreover, for every $a \in \mathfrak A$, we have 
$$\Psi(\Phi(a)^*)\Psi(\Phi(a)) = \Psi(\Phi(a)^* \Phi(a)).$$
This further implies that $\Phi(\mathfrak A) \subset \mult(\Psi)$.
\end{proof}

Let $\Gamma$ be a countable discrete group and $\mu \in \Prob(\Gamma)$ an admissible probability measure. Denote by $(B, \nu_B)$ the $(\Gamma, \mu)$-Furstenberg--Poisson boundary. Set $M = \rL(\Gamma)$ and $\mathscr B = \rL(\Gamma \curvearrowright B)$. Denote by $\rE_{B} : \mathscr B \to \rL^\infty(B, \nu_B)$ the canonical $\Gamma$-equivariant faithful normal conditional expectation.

Let $M \subset \mathscr M$ be an inclusion of von Neumann algebras. Let $\widehat\Theta: \mathscr M \to \mathscr B$ be a faithful normal ucp map such that $\widehat\Theta|_M = \id_M$. Then $\Theta = \rE_B \circ \widehat \Theta : \mathscr M \to \rL^\infty(B, \nu_B)$ is a $\Gamma$-equivariant faithful normal ucp map.

\begin{lem}\label{lem-useful2}
Keep the same notation as above. We have 
$$\Theta(\mult(\Theta)) \rtimes \Gamma \subset \widehat\Theta(\mult(\widehat\Theta)) \subset \rL^\infty(B, \nu_B) \rtimes \Gamma =\mathscr B.$$
In particular, if $\Theta(\mult(\Theta)) = \rL^\infty(B, \nu_B)$, then $\widehat\Theta(\mult(\widehat\Theta)) = \mathscr B$. In that case, we may regard $\mathscr B \subset \mathscr M$ as a von Neumann subalgebra and $\widehat \Theta : \mathscr M \to \mathscr B$ as a faithful normal conditional expectation.
\end{lem}

\begin{proof}
By definition, $\Theta|_{\mult(\Theta)} : \mult(\Theta) \to \rL^\infty(B, \nu_B)$ is a $\Gamma$-equivariant unital normal $\ast$-isomorphism and $\widehat\Theta|_{\mult(\widehat\Theta)} : \mult(\widehat\Theta) \to \mathscr B$ is a unital normal $\ast$-isomorphism such that $M \subset  \widehat\Theta(\mult(\widehat\Theta))\subset \mathscr B$. Since $\rE_B \circ \widehat \Theta = \Theta$ and since $\rE_B$ is faithful, Lemma \ref{lem-useful1} implies that $\widehat{\Theta} |_{\mult(\Theta)} : \mult(\Theta) \to \mathscr B$ is a unital normal $\ast$-isomorphism and $\widehat \Theta(\mult(\Theta)) \subset \mult(\rE_B)$. Then we have $\mult(\Theta) \subset \mult(\widehat\Theta)$ and so $M \vee \mult(\Theta) \subset \mult(\widehat \Theta)$. Since $\mult(\rE_B) = \rL^\infty(B, \nu_B) \subset \mathscr B$, it follows that $\widehat\Theta(\mult(\Theta)) \subset \rL^\infty(B, \nu_B) \subset \mathscr B$ and that $\Theta|_{\mult(\Theta)} = \widehat \Theta|_{\mult(\Theta)}$. Then we have $$\Theta(\mult(\Theta)) \rtimes \Gamma \subset \widehat\Theta(\mult(\widehat\Theta)) \subset \rL^\infty(B, \nu_B) \rtimes \Gamma = \mathscr B.$$

Moreover, assume that $\Theta(\mult(\Theta)) = \rL^\infty(B, \nu_B)$. Then we necessarily have $\widehat\Theta(\mult(\widehat\Theta)) = \mathscr B$. Define the unital normal $\ast$-isomorphism $\iota = (\widehat \Theta|_{\mult(\widehat \Theta)})^{-1} : \mathscr B \to \mult(\widehat\Theta)$. Then $\iota \circ \Theta : \mathscr M \to \mult(\widehat \Theta)$ is a faithful normal conditional expectation.
\end{proof}

\begin{proof}[Proof of Theorem \ref{thm-NCNS}]
Keep the same notation as above with $M = \rL(\Gamma)$ and $\mathscr B = \rL(\Gamma \curvearrowright B)$. Assume moreover that the $(\Gamma, \mu)$-Furstenberg--Poisson boundary $(B, \nu_B)$ is $\mu$-uniquely stationary. We still denote by $B$ the compact metrizable model and we assume that $B = \supp(\nu_B)$. Then the identity map $\id_{\rC(B)} : \rC(B) \to \rL^\infty(B, \nu_B)$ is the unique $\Gamma$-equivariant ucp map (see \cite[Corollary VI.2.10]{Ma91} and \cite[Theorem 3.4]{HK21}). Denote by $\lambda : \Gamma \to \mathscr U(M) : \gamma \mapsto \lambda(\gamma)$ the left regular representation. We naturally identify $\rL^2(M) = \ell^2(\Gamma)$. Denote by $M'$ the commutant of $M$ in $\rB(\rL^2(M))$. Since $M \subset \mathscr M$, we may regard $\rL^2(\mathscr M)$ as a Hilbert left $M$-module. Then \cite[Proposition 8.2.3]{AP14} implies that there exists a projection $e \in M' \ovt \rB(\ell^2)$ and a unitary mapping $V : \rL^2(\mathscr M) \to e(\rL^2(M) \otimes \ell^2)$ such that $V(a \xi) = (a \otimes 1) V(\xi)$ for every $a \in M$ and every $\xi \in \rL^2(\mathscr M)$. Then the mapping 
$$\pi : \rB(\rL^2(M)) \to \rB(\rL^2(\mathscr M)) : T \mapsto V^* e(T \otimes 1) eV$$ is a normal ucp map such that $\pi(a T b) = a \pi(T) b$ for all $T \in \rB(\rL^2(M))$ and all $a, b \in M$.

Since $\mathscr M \subset \rB(\rL^2(\mathscr M))$ is amenable, there exists a (possibly non normal) conditional expectation $\rE : \rB(\rL^2(\mathscr M)) \to \mathscr M$. Choose a point $b \in B$ and define the $\Gamma$-equivariant unital $\ast$-homomorphism $\theta : \rC(B) \to \ell^\infty(\Gamma)  : F \mapsto (\gamma \mapsto F(\gamma b))$. Regard $\ell^\infty(\Gamma) \subset \rB(\ell^2(\Gamma)) = \rB(\rL^2(M))$ and define the ucp map $\iota = \rE \circ \pi \circ \theta : \rC(B) \to \mathscr M$. Then for every $F \in \rC(B)$ and every $\gamma \in \Gamma$, we have 
$$\iota(F \circ \gamma^{-1}) = \rE(\pi(\theta(F) \circ \gamma^{-1})) = \rE(\pi( \lambda(\gamma) \theta(F) \lambda(\gamma)^\ast)) = \lambda(\gamma) \iota(F) \lambda(\gamma)^*.$$
Therefore, the ucp map $\iota : \rC(B) \to \mathscr M$ is $\Gamma$-equivariant with respect to the conjugation action $\Gamma \curvearrowright \mathscr M$.

By composition, $\Theta \circ \iota : \rC(B) \to \rL^\infty(B, \nu_B)$ is a $\Gamma$-equivariant ucp map. Then we have $\Theta \circ \iota = \id_{\rC(B)}$. Since $\Theta : \mathscr M \to \rL^\infty(B, \nu_B)$ is faithful, Lemma \ref{lem-useful1} implies that $\iota : \rC(B) \to \mathscr M$ is a unital $\ast$-homomorphism and $\iota(\rC(B)) \subset \mult(\Theta)$. Set $\psi = \nu_B \circ \Theta \in \mathscr M_\ast$ and note that $\psi$ is a $\mu$-stationary faithful normal state. Since $\psi \circ \iota = \nu_B$, we may uniquely extend $\iota : \rL^\infty(B, \nu_B) \to \mathscr M$ to a faithful normal unital $\ast$-homomorphism such that $\Theta \circ \iota = \id_{\rL^\infty(B, \nu_B)}$. Then we have $\iota(\rL^\infty(B, \nu_B)) \subset \mult(\Theta)$, $\rL^\infty(B, \nu_B) = \Theta(\iota(\rL^\infty(B, \nu_B)))$ and so $\Theta(\mult(\Theta)) = \rL^\infty(B, \nu_B)$. The moreover part follows from Lemma \ref{lem-useful2}.
\end{proof}

\begin{proof}[Proof of Corollary \ref{cor-NCNS}]
Set $M = \rL(\Gamma)$ and $\mathscr B = \rL(\Gamma \curvearrowright B)$. Assume that $\Gamma$ is infinite icc and that the $(\Gamma, \mu)$-Furstenberg--Poisson boundary $(B, \nu_B)$ is $\mu$-uniquely stationary. Denote by $\lambda : \Gamma \to \mathscr U(M) : \gamma \mapsto \lambda(\gamma)$ the left regular representation. Let $\widehat \Theta : \mathscr M \to \mathscr B$ be a normal ucp map such that $\widehat\Theta|_{M} = \id_{M}$ as given by Condition $(\rm iii)$. Denote by $p \in \mathscr M$ the support projection of $\widehat \Theta : \mathscr M \to \mathscr B$. Then for $\gamma \in \Gamma$, we have $\widehat\Theta(\lambda(\gamma) p \lambda(\gamma)^*) = \lambda(\gamma) \widehat\Theta(p) \lambda(\gamma)^* = 1$ and so $p \leq \lambda(\gamma) p \lambda(\gamma)^*$. Since this holds for every $\gamma \in \Gamma$, we infer that $p \in \rL(\Gamma)'\cap \mathscr M$. Condition $(\rm i)$ further implies that $p = 1$. Thus, $\widehat \Theta : \mathscr M \to \mathscr B$ is a faithful normal ucp map. By Condition $(\rm ii)$ and Theorem \ref{thm-NCNS}, we infer that $\mathscr B \subset \mathscr M$ can be regarded as a von Neumann subalgebra and $\widehat\Theta : \mathscr M \to \mathscr B$ as a faithful normal conditional expectation. Then Condition $(\rm iv)$ finally implies that $\mathscr B = \mathscr M$ and so $\widehat\Theta : \mathscr M \to \mathscr B$ is a unital normal onto $\ast$-isomorphism. 
\end{proof}

\begin{proof}[Proof of Theorem \ref{thm-CRC}]
Consider the inclusion $M \subset \mathscr B_1$. Since $M$ is a type ${\rm II_1}$ factor, \cite[Proposition 2.7]{DP20} implies that $M' \cap \mathscr B_1 = \C 1$. By assumption, $\mathscr B_1$ is amenable and $\Phi : \mathscr B_1 \to \mathscr B_2$ is a normal ucp map such that $\Phi|_M = \id_M$. By \cite[Theorem 4.1]{DP20}, for every intermediate von Neumann subalgebra $M \subset \mathscr N \subset \mathscr B_1$ for which there exists a normal conditional expectation $\rE : \mathscr B_1 \to \mathscr N$, we have $\mathscr N = \mathscr B_1$. We may apply Corollary \ref{cor-NCNS} to the inclusion $M \subset \mathscr B_1$ to conclude that $\Phi : \mathscr B_1 \to \mathscr B_2$ is a unital normal onto $\ast$-isomorphism. By \cite[Corollary F]{Ho21} and \cite[Theorem B]{BH22}, we infer that $\rk_{\R}(G_1) = \rk_{\R}(G_2)$.
\end{proof}

\bibliographystyle{plain}

\end{document}